\documentclass[10pt]{amsart}
\usepackage{amssymb,MnSymbol}
\usepackage{amsthm,amsmath}
\usepackage{cite}

\title[Bisymmetric polynomial functions]{A classification of bisymmetric polynomial functions over integral domains of characteristic zero}

\author{Jean-Luc Marichal}
\address{Mathematics Research Unit, FSTC, University of Luxembourg, 6, rue Coudenhove-Kalergi, L-1359 Luxembourg, Luxembourg}
\email{jean-luc.marichal[at]uni.lu}

\author{Pierre Mathonet}
\address{University of Li\`ege, Department of Mathematics, Grande Traverse, 12 - B37, B-4000 Li\`ege, Belgium}
\email{p.mathonet[at]ulg.ac.be }

\date{March 23, 2012}

\begin{document}

\theoremstyle{plain}
\newtheorem{theorem}{Theorem}
\newtheorem{lemma}[theorem]{Lemma}
\newtheorem{proposition}[theorem]{Proposition}
\newtheorem{corollary}[theorem]{Corollary}
\newtheorem{fact}[theorem]{Fact}
\newtheorem{claim}[theorem]{Claim}
\newtheorem*{main}{Main Theorem}

\theoremstyle{definition}
\newtheorem{definition}[theorem]{Definition}
\newtheorem{example}[theorem]{Example}

\theoremstyle{remark}
\newtheorem{remark}{Remark}
\newtheorem{step}{Step}

\newcommand{\N}{\mathbb{N}}
\newcommand{\R}{\mathbb{R}}
\newcommand{\Z}{\mathbb{Z}}
\newcommand{\Q}{\mathbb{Q}}
\newcommand{\C}{\mathbb{C}}
\newcommand{\Id}{\mathrm{Id}}
\newcommand{\Vspace}{\vspace{2ex}}
\newcommand{\bfx}{\mathbf{x}}
\newcommand{\bfy}{\mathbf{y}}
\newcommand{\bfh}{\mathbf{h}}
\newcommand{\bfe}{\mathbf{e}}
\newcommand{\p}{\mathbf{p}}
\newcommand{\bfp}{\mathbf{p}}
\newcommand{\Imp}{\mathcal{I}}
\newcommand{\Sy}{{\mathcal{S}}}
\newcommand{\A}{{\mathcal{A}}}
\newcommand{\bfo}{\boldsymbol{0}}
\newcommand{\bfalpha}{\boldsymbol{\alpha}}
\newcommand{\bfbeta}{\boldsymbol{\beta}}
\newcommand{\bfgamma}{\boldsymbol{\gamma}}

\begin{abstract}
We describe the class of $n$-variable polynomial functions that satisfy Acz\'el's bisymmetry property over an arbitrary integral domain of characteristic zero with identity.
\end{abstract}



\maketitle

\section{Introduction}

Let $\mathcal{R}$ be an integral domain of characteristic zero (hence $\mathcal{R}$ is infinite) with identity and let $n\geqslant 1$ be an integer. In this paper we provide a complete
description of all the $n$-variable polynomial functions over $\mathcal{R}$ that satisfy the (Acz\'el) bisymmetry property. Recall that a function $f\colon\mathcal{R}^n\to\mathcal{R}$ is \emph{bisymmetric} if the $n^2$-variable mapping
$$
\big(x_{11},\ldots,x_{1n};\ldots ;x_{n1},\ldots,x_{nn}\big)\mapsto f\big(f(x_{11},\ldots,x_{1n}),\ldots,f(x_{n1},\ldots,x_{nn})\big)
$$
does not change if we replace every $x_{ij}$ by $x_{ji}$.

The bisymmetry property for $n$-variable real functions goes back to Acz\'el \cite{Acz46,Acz48}. It has been investigated since then in the theory of functional equations by several authors, especially in characterizations of mean functions and some of their extensions (see, e.g., \cite{AczDho89,CouLeh11,FodMar97,GraMarMesPap09}). This property is also studied in algebra where it is called \emph{mediality}. For instance, an algebra $(A,f)$ where $f$ is a bisymmetric binary operation is called a \emph{medial groupoid} (see, e.g., \cite{JezKep83,JezKep83b,Sou71}).

We now state our main result, which provides a description of the possible bisymmetric polynomial functions from $\mathcal{R}^n$ to $\mathcal{R}$. Let $\mathrm{Frac}(\mathcal{R})$ denote the fraction field of $\mathcal{R}$ and let $\N$ be the set of nonnegative integers. For any $n$-tuple $\mathbf{x} = (x_1,\ldots,x_n)$, we set $|\bfx|=\sum_{i=1}^nx_i$.

\begin{main}
A polynomial function $P\colon \mathcal{R}^n\to \mathcal{R}$ is bisymmetric if and only if it is
\begin{itemize}
\item[$(i)$] univariate, or

\item[$(ii)$] of degree $\leqslant 1$, that is, of the form
$$
P(\bfx)=a_0+\sum_{i=1}^n a_i\, x_i\, ,
$$
where $a_i\in\mathcal{R}$ for $i=0,\ldots,n$, or

\item[$(iii)$] of the form
$$
P(\bfx)=a\prod_{i=1}^n(x_i+b)^{\alpha_i}-b\, ,
$$
where $a\in \mathcal{R}$, $b\in \mathrm{Frac}(\mathcal{R})$, and $\bfalpha\in \N^n$ satisfy $a b^k\in \mathcal{R}$ for $k=1,\ldots,|\bfalpha|-1$ and $a b^{|\bfalpha|}-b\in \mathcal{R}$.
\end{itemize}
\end{main}

The following example, borrowed from \cite{MarMat11}, gives a polynomial function of class $(iii)$ for which $b\notin\mathcal{R}$.

\begin{example}
The third-degree polynomial function $P\colon\Z^3\to\Z$ defined on the ring $\Z$ of integers by
\[
P(x_1,x_2,x_3)=9\, x_1x_2x_3+3\, (x_1x_2+x_2x_3+x_3x_1)+x_1+x_2+x_3
\]
is bisymmetric since it is the restriction to $\Z$ of the bisymmetric polynomial function $Q\colon\Q^3\to\Q$ defined on the field $\Q$ of
rationals by
\[
Q(x_1,x_2,x_3)=9\prod_{i=1}^3\Big(x_i+\frac{1}{3}\Big)-\frac{1}{3}\, .
\]
\end{example}

Since polynomial functions usually constitute the most basic functions, the problem of describing the class of bisymmetric polynomial functions is quite natural. On this subject it is noteworthy that a description of the class of bisymmetric lattice polynomial functions over bounded chains and more generally over distributive lattices has been recently obtained \cite{CouLeh11,BehCouKeaLehSze11} (there bisymmetry is called self-commutation), where a lattice polynomial function is a function representable by combinations of variables and constants using the fundamental lattice operations $\wedge$ and $\vee$.

From the Main Theorem we can derive the following test to determine whether a non-univariate polynomial function $P\colon \mathcal{R}^n\to \mathcal{R}$ of degree $p\geqslant 2$ is bisymmetric. For $k\in\{p-1,p\}$, let $P_k$ be the homogenous polynomial function obtained from $P$ by considering the terms of degree $k$ only. Then $P$ is bisymmetric if and only if $P_p$ is a monomial and $P_p(\bfx)=P(\bfx-b\mathbf{1})+b$, where $\mathbf{1}=(1,\ldots,1)$ and $b=P_{p-1}(\mathbf{1})/(p\, P_p(\mathbf{1}))$.

Note that the Main Theorem does not hold for an infinite integral domain $\mathcal{R}$ of characteristic $r>0$. As a counterexample, the bivariate polynomial function $P(x_1,x_2)=x_1^r+x_2^r$ is bisymmetric.

In the next section we provide the proof of the Main Theorem, assuming first that $\mathcal{R}$ is a field and then an integral domain.

\section{Technicalities and proof of the Main Theorem}

We observe that the definition of $\mathcal{R}$ enables us to identify the ring $\mathcal{R}[x_1,\ldots,x_n]$ of polynomials of $n$ indeterminates over $\mathcal{R}$ with the ring of polynomial functions of $n$ variables from $\mathcal{R}^n$ to $\mathcal{R}$.

It is a straightforward exercise to show that the $n$-variable polynomial functions given in the Main Theorem are bisymmetric. We now show that no other $n$-variable polynomial function is bisymmetric. We first consider the special case when $\mathcal{R}$ is a field. We then prove the Main Theorem in the general case (i.e., when $\mathcal{R}$ is an integral domain of characteristic zero with identity).

Unless stated otherwise, we henceforth assume that $\mathcal{R}$ is a field of characteristic zero. Let $p\in\N$ and let $P\colon \mathcal{R}^n\to\mathcal{R}$ be a polynomial function of degree $p$. Thus $P$ can be written in the form
$$
P(\bfx)=\sum_{|\bfalpha|\leqslant p}c_{\bfalpha}\,\bfx^{\bfalpha},\quad \mbox{with}~ \bfx^{\bfalpha}=\prod_{i=1}^nx_i^{\alpha_i}\, ,
$$
where the sum is taken over all $\bfalpha\in\N^n$ such that $|\bfalpha|\leqslant p$.

The following lemma, which makes use of formal derivatives of polynomial functions, will be useful as we continue.

\begin{lemma}
For every polynomial function $B\colon\mathcal{R}^n\to\mathcal{R}$ of degree $p$ and every $\bfx_0,\bfy_0\in\mathcal{R}^n$, we have
\begin{equation}\label{eq:eq4}
B(\bfx_0+\bfy_0)=\sum_{|\bfalpha|\leqslant p}\frac{\bfy_0^{\bfalpha}}{\bfalpha!}\,(\partial^{\bfalpha}_\bfx B)(\bfx_0)\, ,
\end{equation}
where $\partial_{\bfx}^{\bfalpha}=\partial_{x_1}^{\alpha_1}\cdots\,\partial_{x_n}^{\alpha_n}$ and $\bfalpha!=\alpha_1!\cdots\,\alpha_n!$.
\end{lemma}

\begin{proof}
It is enough to prove the result for monomial functions since both sides of (\ref{eq:eq4}) are additive on the function $B$. We then observe that for a monomial function $B(\bfx)=c\,\bfx^{\bfbeta}$ the identity (\ref{eq:eq4}) reduces to the multi-binomial theorem.
\end{proof}

As we will see, it is useful to decompose $P$ into its homogeneous components, that is, $P=\sum_{k=0}^pP_k$, where
$$
P_k(\bfx)=\sum_{|\bfalpha|=k}c_{\bfalpha}\,\bfx^{\bfalpha}
$$
is the unique homogeneous component of degree $k$ of $P$. In this paper the homogeneous component of degree $k$ of a polynomial function $R$ will often be denoted by $[R]_k$.

Since $P_p\neq 0$, the polynomial function $Q=P-P_p$, that is
$$
Q(\bfx)=\sum_{|\bfalpha|< p}c_{\bfalpha}\,\bfx^{\bfalpha},
$$
is of degree $q<p$ and its homogeneous component $[Q]_q$ of degree $q$ is $P_q$.

We now assume that $P$ is a bisymmetric polynomial function. This means that the polynomial identity
\begin{equation}\label{eq:bisym}
P\big(P(\mathbf{r}_1),\ldots,P(\mathbf{r}_n)\big)-P\big(P(\mathbf{c}_1),\ldots,P(\mathbf{c}_n)\big)=0
\end{equation}
holds for every $n\times n$ matrix
\begin{equation}\label{eq:matrix}
X=
\begin{pmatrix}
x_{11} & \cdots & x_{1n}\\
\vdots & \ddots & \vdots\\
x_{n1} & \cdots & x_{nn}
\end{pmatrix}
\in \mathcal{R}^n_n\, ,
\end{equation}
where $\mathbf{r}_i=(x_{i1},\ldots,x_{in})$ and $\mathbf{c}_j= (x_{1j},\ldots,x_{nj})$ denote its $i$th row and $j$th column, respectively. Since all the polynomial functions of degree $\leqslant 1$ are bisymmetric, we may (and henceforth do) assume that $p\geqslant 2$.

From the decomposition $P=P_p+Q$ it follows that
$$
P\big(P(\mathbf{r}_1),\ldots,P(\mathbf{r}_n)\big)=P_p\big(P(\mathbf{r}_1),\ldots,P(\mathbf{r}_n)\big)+Q\big(P(\mathbf{r}_1),\ldots,P(\mathbf{r}_n)\big),
$$
where $Q(P(\mathbf{r}_1),\ldots,P(\mathbf{r}_n))$ is of degree $p\, q$.

To obtain necessary conditions for $P$ to be bisymmetric, we will equate the homogeneous components of the same degree $>p\, q$ of both sides of (\ref{eq:bisym}). By the previous observation this amounts to considering the equation
\begin{equation}\label{eq:bisympp}
\big[P_p\big(P(\mathbf{r}_1),\ldots,P(\mathbf{r}_n)\big)-P_p\big(P(\mathbf{c}_1),\ldots,P(\mathbf{c}_n)\big)\big]_d=0\, ,\quad \mbox{for $~p\, q< d\leqslant p^2$}.
\end{equation}
By applying (\ref{eq:eq4}) to the polynomial function $P_p$ and the $n$-tuples
$$
\bfx_0=(P_p(\mathbf{r}_1),\ldots,P_p(\mathbf{r}_n))\quad\mbox{and}\quad\bfy_0=(Q(\mathbf{r}_1),\ldots,Q(\mathbf{r}_n)),
$$
we obtain
\begin{equation}\label{eq:eq6}
P_p(P(\mathbf{r}_1),\ldots,P(\mathbf{r}_n)) ~=~ \sum_{|\bfalpha|\leqslant p}\frac{\bfy_0^{\bfalpha}}{\bfalpha!}\,\partial^{\bfalpha}_\bfx P_p(\bfx_0)
\end{equation}
and similarly for $P_p(P(\mathbf{c}_1),\ldots,P(\mathbf{c}_n))$. We then observe that in the sum of (\ref{eq:eq6}) the term corresponding to a fixed $\bfalpha$ is either zero or of degree $$q\,|\bfalpha|+(p-|\bfalpha|)\,p=p^2-(p-q)\,|\bfalpha|$$ and its homogeneous component of highest degree is obtained by ignoring the components of degrees $<q$ in $Q$, that is, by replacing $\bfy_0$ by $(P_q(\mathbf{r}_1),\ldots,P_q(\mathbf{r}_n))$.

Using (\ref{eq:bisympp}) with $d=p^2$, which leads us to consider the terms in (\ref{eq:eq6}) for which $|\bfalpha|=0$, we obtain
\begin{equation}\label{eq:Pp}
P_p(P_p(\mathbf{r}_1),\ldots,P_p(\mathbf{r}_n))-P_p(P_p(\mathbf{c}_1),\ldots,P_p(\mathbf{c}_n))=0.
\end{equation}
Thus, we have proved the following claim.

\begin{claim}\label{claim:111}
The polynomial function $P_p$ is bisymmetric.
\end{claim}

We now show that $P_p$ is a monomial function.

\begin{proposition}\label{prop:homog}
Let $H\colon\mathcal{R}^n\to\mathcal{R}$ be a bisymmetric polynomial function of degree $p\geqslant 2$. If $H$ is homogeneous, then it is a monomial function.
\end{proposition}

\begin{proof}
Consider a bisymmetric homogeneous polynomial $H\colon\mathcal{R}^n\to\mathcal{R}$ of degree $p\geqslant 2$. There is nothing to prove if $H$ depends on one variable only. Otherwise, assume for the sake of a contradiction that $H$ is the sum of at least two monomials of degree $p$, that is,
\[
H(\bfx)=a\, \mathbf{x}^{\bfalpha}+b\, \mathbf{x}^{\bfbeta}+\sum_{|\bfgamma|=p}c_{\bfgamma}\,\bfx^{\bfgamma},
\]
where $a\,b\neq 0$ and $|\bfalpha|=|\bfbeta|=p$.
Using the lexicographic order $\preccurlyeq$ over $\N^n$, we can assume that $\bfalpha\succ\bfbeta\succ\bfgamma$. Applying the bisymmetry property of $H$ to the $n\times n$ matrix whose $(i,j)$-entry is $x_iy_j$, we obtain
\[
H(\bfx)^p\, H(\bfy^p)=H(\bfy)^p\, H(\bfx^p),
\]
where $\bfx^p=(x_1^p,\ldots,x_n^p)$. Regarding this equality as a polynomial identity in $\bfy$ and then equating the coefficients of its monomial terms with exponent $p\,\bfalpha$, we obtain
\begin{equation}\label{eq:sdf79}
H(\bfx)^p= a^{p-1}\, H(\bfx^p).
\end{equation}
Since $\mathcal{R}$ is of characteristic zero, there is a nonzero monomial term with exponent $(p-1)\,\bfalpha+\bfbeta$ in the left-hand side of (\ref{eq:sdf79}) while there is no such term in the right-hand side since $p\,\bfalpha\succ (p-1)\,\bfalpha+\bfbeta\succ p\,\bfbeta$ (since $p\geqslant 2$). Hence a contradiction.
\end{proof}


The next claim follows immediately from Proposition~\ref{prop:homog}.

\begin{claim}\label{claim:homog}
$P_p$ is a monomial function.
\end{claim}

By Claim~\ref{claim:homog} we can (and henceforth do) assume that there exist $c\in\mathcal{R}\setminus\{0\}$ and $\bfgamma\in\N^n$, with $|\bfgamma|=p$, such that
\begin{equation}\label{eq:8sadsf}
P_p(\bfx)=c\,\bfx^{\bfgamma}.
\end{equation}

A polynomial function $F\colon\mathcal{R}^n\to\mathcal{R}$ is said to \emph{depend on} its $i$th variable $x_i$ (or $x_i$ is \emph{essential} in $F$) if $\partial_{x_i}F\neq 0$. The following claim shows that $P_p$ determines the essential variables of $P$.

\begin{claim}\label{claim:8243}
If $P_p$ does not depend on the variable $x_j$, then $P$ does not depend on $x_j$.
\end{claim}

\begin{proof}
Suppose that $\partial_{x_j}P_p=0$ and fix $i\in\{1,\ldots,n\}$, $i\neq j$, such that $\partial_{x_i}P_p\not=0$. By taking the derivative of both sides of (\ref{eq:bisym}) with respect to $x_{ij}$, the $(i,j)$-entry of the matrix $X$ in (\ref{eq:matrix}), we obtain
\begin{equation}\label{eq:sd7f5}
(\partial_{x_i}P)(P(\mathbf{r}_1),\ldots,P(\mathbf{r}_n))(\partial_{x_j}P)(\mathbf{r}_i)=(\partial_{x_j}P)(P(\mathbf{c}_1),\ldots,P(\mathbf{c}_n))(\partial_{x_i}P)(\mathbf{c}_j).
\end{equation}
Suppose for the sake of a contradiction that $\partial_{x_j}P\neq 0$. Thus, neither side of (\ref{eq:sd7f5}) is the zero polynomial. Let $R_j$ be the homogeneous component of $\partial_{x_j}P$ of highest degree and denote its degree by $r$. Since $P_p$ does not depend on $x_j$, we must have $r< p-1$. Then the homogeneous component of highest degree of the left-hand side in (\ref{eq:sd7f5}) is given by
\[
(\partial_{x_i}P_p)(P_p(\mathbf{r}_1),\ldots,P_p(\mathbf{r}_n))\,R_j(\mathbf{r}_i)
\]
and is of degree $p(p-1)+r$. But the right-hand side in (\ref{eq:sd7f5}) is of degree at most $rp+p-1=(r+1)(p-1)+r<p(p-1)+r$, since $r< p-1$ and $p\geqslant 2$. Hence a contradiction. Therefore $\partial_{x_j}P=0$.
\end{proof}

We now give an explicit expression for $P_q=[P-P_p]_q$ in terms of $P_p$. We first present an equation that is satisfied by $P_q$.

\begin{claim}\label{claim:8s7f}
$P_q$ satisfies the equation
\begin{equation}\label{eqbeta1}
\sum_{i=1}^n P_q({\bf r}_i)(\partial_{x_i} P_p)(P_p({\bf r}_1),\ldots,P_p({\bf r}_n))= \sum_{i=1}^n P_q({\bf c}_i)(\partial_{x_i} P_p)(P_p({\bf c}_1),\ldots,P_p({\bf c}_n))
\end{equation}
for every matrix $X$ as defined in (\ref{eq:matrix}).
\end{claim}

\begin{proof}
By (\ref{eq:Pp}) and (\ref{eq:8sadsf}) we see that the left-hand side of (\ref{eq:bisympp}) for $d=p^2$ is zero. Therefore, the highest degree terms in the sum of (\ref{eq:eq6}) are of degree $p^2-(p-q)>p\, q$ (because $(p-1)(p-q)>0$) and correspond to those $\bfalpha\in\N^n$ for which $|\bfalpha|=1$. Collecting these terms and then considering only the homogeneous component of highest degree (that is, replacing $Q$ by $P_q$), we see that the identity (\ref{eq:bisympp}) for $d=p^2-(p-q)$ is precisely (\ref{eqbeta1}).
\end{proof}

\begin{claim}\label{claim:equiv}
We have
\begin{equation}\label{Pqform}
P_q(\bfx)=\frac{P_q(\boldsymbol{1})}{c\, p}\, P_p(\bfx)\,\sum\limits_{j=1}^n\frac{\gamma_j}{x_j^{p-q}}\, .
\end{equation}
Moreover, $P_q=0$ if there exists $j\in\{1,\ldots,n\}$ such that $0<\gamma_j < p-q$.
\end{claim}

\begin{proof}
Considering Eq.~(\ref{eqbeta1}) for a matrix $X$ such that ${\mathbf r}_i=\bfx$ for $i=1,\ldots,n$, we obtain
\[
c\, p\, P_q(\bfx)\, P_p(\bfx)^{p-1}=P_q(\boldsymbol{1})\sum_{i=1}^nx_i^q(\partial_{x_i}P_p)(c\, x_1^p,\ldots,c\, x_n^p).
\]
Since $\partial_{x_i}P_p(\bfx)=\gamma_i\, P_p(\bfx)/x_i$, 
the previous equation becomes
\begin{equation}\label{eq:gfhfs456}
c\, p\, P_q(\bfx)\, P_p(\bfx)^{p-1}=P_q(\boldsymbol{1})\,P_p(\bfx)^p\,\sum_{i=1}^n \frac{\gamma_i}{x_i^{p-q}}
\end{equation}
from which Eq.~(\ref{Pqform}) follows. Now suppose that $P_q\neq 0$ and let $j\in\{1,\ldots,n\}$. Comparing the lowest degrees in $x_j$ of both sides of (\ref{eq:gfhfs456}), we obtain
$$
(p-1)\,\gamma_j\leqslant
\begin{cases}
p\,\gamma_j-p+q\, , & \mbox{if $\gamma_j\neq 0$},\\
p\,\gamma_j\, , & \mbox{if $\gamma_j= 0$}.
\end{cases}
$$
Therefore, we must have $\gamma_j=0$ or $\gamma_j\geqslant p-q$, which ensures that the right-hand side of (\ref{Pqform}) is a polynomial.
\end{proof}

If $\varphi\colon\mathcal{R}\to\mathcal{R}$ is a bijection, we can
associate with every function $f\colon\mathcal{R}^n\to\mathcal{R}$ its \emph{conjugate} $\varphi(f)\colon\mathcal{R}^n\to\mathcal{R}$ defined by
\[
\varphi(f)(x_1,\ldots,x_n)=\varphi^{-1}\big(f(\varphi(x_1),\ldots,\varphi(x_n))\big).
\]
It is clear that $f$ is bisymmetric if and only if so is $\varphi(f)$. We then have the following fact.

\begin{fact}\label{prop:conj}
The class of $n$-variable bisymmetric functions is stable under the action of conjugation.
\end{fact}

Since the Main Theorem involves polynomial functions over a ring, we will only consider conjugations given by translations $\varphi_b(x)=x+b$.

We now show that it is always possible to conjugate $P$ with an appropriate translation $\varphi_b$ to eliminate the terms of degree $p-1$ of the resulting polynomial function $\varphi_b(P)$.

\begin{claim}\label{claimcor1}
There exists $b\in R$ such that $\varphi_b(P)$ has no term of degree $p-1$.
\end{claim}

\begin{proof}
If $q<p-1$, we take $b=0$. If $q=p-1$, then using (\ref{eq:eq4}) with $\bfy_0=b\mathbf{1}$, we get
$$
\big[\varphi_b(P)\big]_{p-1}=P_{p-1}+b\, \sum_{i=1}^n\partial_{x_i}P_p\, .
$$
On the other hand, by (\ref{Pqform}) we have
\[
P_{p-1}=\frac{P_{p-1}(\boldsymbol{1})}{c\, p}\,\sum_{i=1}^n\partial_{x_i}P_p\, .
\]
It is then enough to choose $b=-P_{p-1}(\boldsymbol{1})/(c\, p)$ and the result follows.
\end{proof}

We can now prove the Main Theorem for polynomial functions of degree $\leqslant 2$.

\begin{proposition}\label{prop:67sfda}
The Main Theorem is true when $\mathcal{R}$ is a field of characteristic zero and $P$ is a polynomial function of degree $\leqslant 2$.
\end{proposition}

\begin{proof}
Let $P$ be a bisymmetric polynomial function of degree $p\leqslant 2$. If $p\leqslant 1$, then $P$ is in class $(ii)$ of the Main Theorem. If $p=2$, then by Claim~\ref{claimcor1} we see that $P$ is (up to conjugation) of the form
$P(\bfx)=c_2\, x_i\, x_j+c_0$. If $i=j$, then by Claim~\ref{claim:8243} we see that $P$ is a univariate polynomial function, which corresponds to the class $(i)$. If $i\not=j$, then by Claim~\ref{claim:equiv} we have $c_0=0$ and hence $P$ is a monomial (up to conjugation).
\end{proof}

By Proposition~\ref{prop:67sfda} we can henceforth assume that $p\geqslant 3$. We also assume that $P$ is a bivariate polynomial function. The general case will be proved by induction on the number of essential variables of $P$.

\begin{proposition}\label{prop:s8df6}
The Main theorem is true when $\mathcal{R}$ is a field of characteristic zero and $P$ is a bivariate polynomial function.
\end{proposition}

\begin{proof}
Let $P$ be a bisymmetric bivariate polynomial function of degree $p\geqslant 3$. We know that $P_p$ is of the form $P_p(x,y)=c\, x^{\gamma_1} y^{\gamma_2}$. If $\gamma_1\,\gamma_2=0$, then by Claim~\ref{claim:8243} we see that $P$ is a univariate polynomial function, which corresponds to the class $(i)$.

Conjugating $P$, if necessary, we may assume that $P_{p-1}=0$ (by Claim~\ref{claimcor1}) and it is then enough to prove that $P=P_p$ (i.e., $P_q=0$). If $\gamma_1=1$ or $\gamma_2=1$, then the result follows immediately from Claim~\ref{claim:equiv} since $p-q\geqslant 2$. We may therefore assume that $\gamma_1\geqslant 2$ and $\gamma_2\geqslant 2$. We now prove that $P=P_p$ in three steps.

\begin{step}\label{Claim:1}
$P(x,y)$ is of degree $\leqslant\gamma_1$ in $x$ and of degree $\leqslant\gamma_2$  in $y$.
\end{step}

\begin{proof}
We prove by induction on $r\in\{0,\ldots,p-1\}$ that $P_{p-r}(x,y)$ is of degree $\leqslant\gamma_1$ in $x$ and of degree $\leqslant\gamma_2$ in $y$. The result is true by our assumptions for $r=0$ and $r=1$ and is obvious for $r=p$. Considering Eq.~(\ref{eq:bisympp}) for $d=p^2-r>p\, q$, with $\mathbf{r}_1=\mathbf{r_2}=(x,y)$, we obtain
\begin{equation}\label{eq:ident1}
\left[P(x,y)^p\right]_{p^2-r}=\left[P(x,x)^{\gamma_1}\, P(y,y)^{\gamma_2}\right]_{p^2-r}\, .
\end{equation}
Clearly, the right-hand side of (\ref{eq:ident1}) is a polynomial function of degree $\leqslant p\,\gamma_1$ in $x$ and $\leqslant p\,\gamma_2$ in $y$. Using the multinomial theorem, the left-hand side of (\ref{eq:ident1}) becomes
\[
\left[P(x,y)^p\right]_{p^2-r}=\left[\left(\sum_{k=0}^pP_{p-k}(x,y)\right)^p\right]_{p^2-r}=\sum_{\bfalpha\in A_{p,r}} {p\choose\bfalpha}\prod_{k=0}^pP_{p-k}(x,y)^{\alpha_k}\, ,
\]
where
\[
A_{p,r}=\Big\{\bfalpha=(\alpha_0,\ldots,\alpha_p)\in\N^{p+1}:\sum_{k=0}^pk\, \alpha_k=r,\, |\bfalpha|=p\Big\}.
\]
Observing that for every $\bfalpha\in A_{p,r}$ we have $\alpha_k=0$ for $k>r$ and $\alpha_r\not=0$ only if $\alpha_r=1$ and $\alpha_0=p-1$, we can rewrite (\ref{eq:ident1}) as
\[
p\, P_p(x,y)^{p-1}\, P_{p-r}(x,y)=\left[P(x,x)^{\gamma_1} P(y,y)^{\gamma_2}\right]_{p^2-r}-\sum_{\textstyle{\bfalpha\in A_{p,r}\atop \alpha_r=\cdots =\alpha_p=0}} {p \choose \bfalpha}\prod_{k=0}^{r-1}P_{p-k}(x,y)^{\alpha_k}\, .
\]
By induction hypothesis, the right-hand side is of degree $\leqslant p\,\gamma_1$ in $x$ and of degree $\leqslant p\,\gamma_2$  in $y$. The result then follows by analyzing the highest degree terms in $x$ and $y$ of the left-hand side.
\end{proof}

\begin{step}\label{step2}
$P(x,y)$ factorizes into a product $P(x,y)=U(x)\, V(y).$
\end{step}

\begin{proof}
By Step~\ref{Claim:1}, we can write
\[
P(x,y)=\sum_{k=0}^{\gamma_1}x^k\,V_k(y)\, ,
\]
where $V_k$ is of degree $\leqslant\gamma_2$ and $V_{\gamma_1}(y)=\sum_{j=0}^{\gamma_2}c_{{\gamma_2}-j}\, y^j$, with $c_0=c\not=0$ and $c_1=0$ (since $P_{p-1}=0$). Equating the terms of degree $\gamma_1^2$ in $z$ in the identity
\[
P(P(z,t),P(x,y))=P(P(z,x),P(t,y))\, ,
\]
we obtain
$$
V_{\gamma_1}(t)^{\gamma_1}\,V_{\gamma_1}(P(x,y))=V_{\gamma_1}(x)^{\gamma_1}\,V_{\gamma_1}(P(t,y)).
$$
Equating now the terms of degree $\gamma_1\gamma_2$ in $t$ in the latter identity, we obtain
\begin{equation}\label{eq:ident2}
c^{\gamma_1}\, V_{\gamma_1}(P(x,y))=c\, V_{\gamma_1}(x)^{\gamma_1}\, V_{\gamma_1}(y)^{\gamma_2}\, .
\end{equation}
We now show by induction on $r\in\{0,\ldots,\gamma_1\}$ that every polynomial function $V_{\gamma_1-r}$ is a multiple of $V_{\gamma_1}$ (the case $r=0$ is trivial), which is enough to prove the result. To do so, we equate the terms of degree $\gamma_1\gamma_2-r$ in $x$ in (\ref{eq:ident2}) (by using the explicit form of $V_{\gamma_1}$ in the left-hand side). Note that terms with such a degree in $x$ can appear in the expansion of $V_{\gamma_1}(P(x,y))$ only when $P(x,y)$ is raised to the highest power $\gamma_2$ (taking into account the condition $c_1 = 0$ when $r =\gamma_1$). Thus, we obtain
$$
c^{\gamma_1+1}\,\left[\left(\sum_{k=0}^{\gamma_1}x^{\gamma_1-k}\, V_{\gamma_1-k}(y)\right)^{\gamma_2}\right]_{\gamma_1\gamma_2-r}=c\,[V_{\gamma_1}(x)^{\gamma_1}]_{\gamma_1\gamma_2-r}V_{\gamma_1}(y)^{\gamma_2}\, ,
$$
(here the notation $[\cdot]_{\gamma_1\gamma_2-r}$ concerns only the degree in $x$). By computing the left-hand side (using the multinomial theorem as in the proof of Step~\ref{Claim:1}) and using the induction on $r$, we finally obtain an identity of the form
$$
a\, V_{\gamma_1}(y)^{\gamma_2-1}\, V_{\gamma_1-r}(y)=a'\,V_{\gamma_1}(y)^{\gamma_2},\qquad a,a'\in\mathcal{R},\, a\neq 0,
$$
from which the result immediately follows.
\end{proof}

\begin{step}\label{step3}
$U$ and $V$ are monomial functions.
\end{step}

\begin{proof}
Using (\ref{eq:ident2}) with $P(x,y)=U(x)\, V(y)$ and $V_{\gamma_1}=V$, we obtain
\begin{equation}\label{eq:s8fd7af6}
c^{\gamma_1}\,\sum_{j=0}^{\gamma_2}c_{\gamma_2-j}\,(U(x)\,V(y))^j=c\, V(x)^{\gamma_1}\, V(y)^{\gamma_2}.
\end{equation}
Equating the terms of degree $\gamma_2^2$ in $y$ in (\ref{eq:s8fd7af6}), we obtain
\begin{equation}\label{eq:s8d7af6}
c^{\gamma_1+\gamma_2+1}\, U(x)^{\gamma_2}=c^{\gamma_2+1}\, V(x)^{\gamma_1}\, .
\end{equation}
Therefore, (\ref{eq:s8fd7af6}) becomes
$$
\sum_{j=0}^{\gamma_2-1}c_{\gamma_2-j}\,(U(x)\,V(y))^j=0,
$$
which obviously implies $c_k=0$ for $k=1,\ldots,\gamma_2$, which in turn implies $V(x)=c\, x^{\gamma_2}$. Finally, from (\ref{eq:s8d7af6}) we obtain $U(x)=x^{\gamma_1}$.
\end{proof}

\noindent Steps~\ref{step2} and \ref{step3} together show that $P=P_p$, which establishes the proposition.
\end{proof}

Recall that the action of the symmetric group $\mathfrak{S}_n$ on functions from $\mathcal{R}^n$ to $\mathcal{R}$ is defined by
\[
\sigma(f)(x_1,\ldots,x_n)=f(x_{\sigma(1)},\ldots,x_{\sigma(n)}),\qquad\sigma\in\mathfrak{S}_n.
\]
It is clear that $f$ is bisymmetric if and only if so is $\sigma(f)$. We then have the following fact.

\begin{fact}\label{prop:sym}
The class of $n$-variable bisymmetric functions is stable under the action of the symmetric group $\mathfrak{S}_n$.
\end{fact}

Consider also the following action of identification of variables. For $f\colon\mathcal{R}^n\to\mathcal{R}$ and $i<j\in [n]$ we define the function $I_{i,j}f\colon\mathcal{R}^{n-1}\to\mathcal{R}$ as
\[
(I_{i,j}f)(x_1,\ldots,x_{n-1})=f(x_1,\ldots,x_{j-1},x_i,x_{j},\ldots,x_{n-1}).
\]
This action amounts to considering the restriction of $f$ to the ``subspace of equation $x_i=x_j$'' and then relabeling the variables. By Fact~\ref{prop:sym} it is enough to consider the identification of the first and second variables, that is,
\[
(I_{1,2}f)(x_1,\ldots,x_{n-1})=f(x_1,x_1,x_2\ldots,x_{n-1}).
\]

\begin{proposition}\label{prop:ident}
The class of $n$-variable bisymmetric functions is stable under identification of variables.
\end{proposition}

\begin{proof}
To see that $I_{1,2}f$ is bisymmetric, it is enough to apply the bisymmetry of $f$ to the $n\times n$ matrix
\[
\begin{pmatrix}
x_{1,1} & x_{1,1} & \cdots & x_{1,n-1}\\
x_{1,1} & x_{1,1} & \cdots & x_{1,n-1}\\
\vdots & \vdots & \ddots & \vdots\\
x_{n-1,1} & x_{n-1,1} & \cdots & x_{n-1,n-1}
\end{pmatrix}\, .
\]
To see that $I_{i,j}f$ is bisymmetric, we can similarly consider the matrix whose rows $i$ and $j$ are identical and the same for the columns (or use Fact~\ref{prop:sym}).
\end{proof}

We now prove the Main Theorem by using both a simple induction on the number of essential variables of $P$ and the action of identification of variables.

\begin{proof}[Proof of the Main Theorem when $\mathcal{R}$ is a field]
We proceed by induction on the number of essential variables of $P$. By Proposition~\ref{prop:s8df6} the result holds when $P$ depends on $1$ or $2$ variables only. Let us assume that the result also holds when $P$ depends on $n-1$ variables ($n-1\geqslant 2$) and let us prove that it still holds when $P$ depends on $n$ variables. By Proposition~\ref{prop:67sfda} we may assume that $P$ is of degree $p \geqslant 3$. We know that $P_p(\bfx)=c\, \bfx^{\bfgamma}$, where $c\neq 0$ and $\gamma_i>0$ for $i=1,\ldots,n$ (cf.\ Claim~\ref{claim:8243}). Up to a conjugation we may assume that $P_{p-1}=0$ (cf.\ Claim~\ref{claimcor1}). Therefore, we only need to prove that $P=P_p$. Suppose on the contrary that $P-P_p$ has a polynomial function $P_q\neq 0$ as the homogeneous component of highest degree. Then the polynomial function $I_{1,2}\, P$ has $n-1$ essential variables, is bisymmetric (by Proposition~\ref{prop:ident}), has $I_{1,2}\, P_p$ as the homogeneous component of highest degree (of degree $p\geqslant 3$), and has no component of degree $p-1$. By induction hypothesis, $I_{1,2}\, P$ is in class $(iii)$ of the Main Theorem with $b=0$ (since it has no term of degree $p-1$) and hence it should be a monomial function. However, the polynomial function $[I_{1,2}\, P]_q=I_{1,2}\, P_q$ is not zero by (\ref{Pqform}), hence a contradiction.
\end{proof}

\begin{proof}[Proof of the Main Theorem when $\mathcal{R}$ is an integral domain]
Using the identification of polynomials and polynomial functions, we can extend every bisymmetric polynomial function over an integral domain $\mathcal{R}$ with identity to a polynomial function on $\mathrm{Frac}(\mathcal{R})$. The latter function is still bisymmetric since the bisymmetry property for polynomial functions is defined by a set of polynomial equations on the coefficients of the polynomial functions. Therefore, every bisymmetric polynomial function over $\mathcal{R}$ is the restriction to $\mathcal{R}$ of a bisymmetric polynomial function over $\mathrm{Frac}(\mathcal{R})$. We then conclude by using the Main Theorem for such functions.
\end{proof}

\section*{Acknowledgments}

The authors wish to thank J.\ Dasc\u{a}l and E.\ Lehtonen for fruitful discussions. This research is supported by the internal research project F1R-MTH-PUL-12RDO2 of the University of Luxembourg.


\end{document}